\documentclass[12pt]{article}
\usepackage{fancyhdr}
\usepackage{amssymb}
\usepackage{amsmath}
\usepackage{amsthm}
\usepackage{graphicx}
\usepackage{psfrag}
\usepackage{graphics}
\usepackage{amsthm}
\usepackage{xcolor,hyperref}

\input pictex.tex

\newcommand{\esp}{\hspace{0.06cm}}

\newcommand{\R}{\mathbb R}

\newcommand{\Diff}{\mathrm{Diff}}

\newcommand{\var}{\mathrm{var}}

\newcommand{\id}{\mathrm{id}}
\newcommand{\Rplus}{\mathbb{R}_{+}}

\theoremstyle{theorem}
\newtheorem{thmintro}{Theorem}

\newtheorem{thm}{Theorem}[section]
\newtheorem{prop}[thm]{Proposition}
\newtheorem{lem}[thm]{Lemma}

\newtheorem{qs}[thm]{Question}
\theoremstyle{definition}

\theoremstyle{remark}
\newtheorem{rem}[thm]{Remark}
\newtheorem{ex}[thm]{Example}

\begin{document}

\pagestyle{fancy}
\rhead{On the linearization of $C^1$ hyperbolic germs in dimension one}

\date{}
\author{H\'el\`ene Eynard-Bontemps \,\, \& \,\, Andr\'es Navas}

\title{On the failure of linearization for germs of 
$C^1$ hyperbolic vector fields in dimension one}
\maketitle
\begin{footnotesize}
\noindent {\bf H\'el\`ene Eynard-Bontemps} \hfill{\bf Andr\'es Navas}

\noindent helene.eynard-bontemps@univ-grenoble-alpes.fr \hfill{andres.navas@usach.cl}

\noindent Institut Fourier \hfill{ Dpto. de Matem\'aticas y C.C.}

\noindent  Universit\'e Grenoble Alpes \hfill{ Universidad de Santiago de  Chile}

\noindent 100 rue des Math\'ematiques \hfill{Alameda Bernardo O'Higgins 3363}

\noindent 38610 Gi\`eres, France \hfill{Estaci\'on Central, Santiago, Chile} 

\end{footnotesize}
\bigskip
%%%%%%%%%%%%%%%%%%%%%%%%%%%%%%%%%%%%%%%%%%%%%%%%%%%%%%%%%%%%%%%%%%%%%%%

\noindent{\bf Abstract:} We investigate conjugacy classes of germs of hyperbolic 1-dimensional vector fields at the origin in low regularity. We show that the classical linearization theorem of Sternberg strongly fails in this setting by providing  explicit uncountable families of mutually non-conjugate flows with the same multipliers, where conjugacy is considered in the bi-Lipschitz, $C^1$ and $C^{1+ac}$ settings.

\vspace{0.3cm}

\noindent{\bf Keywords:} vector field, diffeomorphism, hyperbolicity, linearization.

\vspace{0.3cm}

\noindent{\bf MSC 2020:} 37D99, 37E05, 37G05.

\vspace{0.2cm}

\noindent{\bf Acknowledgments.} 
We would like to thank \'Etienne Ghys for having asked the question of existence of non-linearizable hyperbolic vector fields of class $C^1$ in relation to 
his work with Grant Cairns on linearization of local $\mathrm{SL}(n)$-actions \cite{CG}. We would also like to thank Jan Kiwi for his interest on the subject, 
as well as Fr\'ed\'eric Le Roux for several clever remarks. Some material of this work arose following a question of Michele Triestino about normalizers of 
flows near a singularity and we thank him for stimulating discussions. 

H. Eynard-Bontemps was detached at the CMM of University of Chile during the elaboration of this work and was partially funded by CNRS and 
by the IRGA project ADMIN of Grenoble INP - Universit\'e Grenoble Alpes. A. Navas was supported by the Fondecyt research project 1220032.

\vspace{0.75cm}

Starting with Poincar\'e, the study of the behavior of a flow dynamics near a hyperbolic singularity or a map near a hyperbolic fixed point has been treated by numerous authors. A basic tool for this is the classical Grobman-Hartman theorem, which establishes that in a neighborhood of such a point, the dynamics is topologically conjugate to that of the corresponding linear system. There is a large amount of literature concerning the regularity of this conjugacy for smooth-enough vector fields and maps. In higher dimension, a problem of resonance between the multipliers may involve a huge loss of regularity even for analytic systems, but the loss of regularity for the conjugacy map is very well controlled in the non-resonant setting; see \cite{stowe} for general results on this in the case of planar maps and \cite{zhang} for a more recent sharper result.

In the 1-dimensional case there is no resonance and, hence, no loss of regularity for the conjugacy. In concrete terms, every $C^r$-germ of orientation-preserving diffeomorphism at a hyperbolic fixed point (\emph{i.e.} a fixed point where the derivative differs from $1$) is locally conjugate to the corresponding linear map on a neighborhood of the point via a $C^r$ local diffeomorphism, where $r$ is any number strictly greater than 1. (Here, $r$ can take non-integer values, and in this case the assumption is that the derivative of order $[r]$ is H\"older continuous with exponent $\{ r \}$.) This statement (which easily extends to hyperbolic flows) is usually referred to as Sternberg's linearization theorem \cite{sternberg}, though this sharp version is due to Yoccoz~\cite{yoccoz} for integer values of $r$, and has been extended to non integer values by many people (see for instance \cite{mann}) and used in several different contexts (see  \cite{DKN} and \cite{N-jam}). The goal of this work is to deeply elaborate on the fact that, for $r=1$, this \emph{smooth} linearization result fails in the 1-dimensional setting.
 
 For simplicity, unless otherwise stated, in this work, we will use the term ``germ" to refer to germs of diffeomorphisms / vector fields on the real line with a fixed point / singularity at the origin. Sternberg himself gave an example of a germ of hyperbolic $C^1$ diffeomorphism that is not $C^1$ linearizable, namely 
\begin{equation}\label{sternberg-ex}
x \mapsto  a x \left( 1-\frac{1}{\log (x)} \right),
\end{equation}
where $a < 1$. Moreover \cite[Exercise 4.1.12]{book} explains how to upgrade this example so that the germ becomes the time-$1$ map of a $C^1$ vector field with absolutely continuous derivative. The first result of this work is a very general version of this phenomenon for vector fields.

\begin{thmintro}
\label{t:A}
Every germ of hyperbolic $C^1$ vector field is contained in a $C^1$-continuous path of such germs, all of 
them with the same multiplier at the origin, and whose time-1 maps are pairwise non bi-Lipschitz conjugate.
\end{thmintro}

Here, by ``multiplier'' of a germ of $C^1$ vector field $X$ we just mean the value~$\lambda := DX(0)$, and hyperbolicity means that this number differs from $0$, so that the derivative at the origin of any non-identity flow map is different from $1$. Indeed, in terms of the flow, this number is also characterized by the relation $Df (0) = e^{\lambda}$, where $f$ is the time-1 map of the flow of $X$ ({\em cf.} Lemma \ref{l:form-vf}). Moreover, it is invariant under bi-Lipschitz conjugacy ({\em cf.} Lemma \ref{l:pres-mult}).

Our techniques allow  to build explicit families of exotic vector fields, and are not only existential as those based, for instance, on Anosov-Katok's method (compare \cite[Proposition 4.15]{BMNR}). Many of the (germs of) vector fields that we exhibit are indeed of class $C^{1+ac}$, where $ac$ stands for {\em absolutely continuous}. In particular, their flows yield diffeomorphisms with absolutely continuous (hence of bounded variation) derivative ({\em cf.} Proposition \ref{p:flow-ac}). As concrete examples, this is the case of the vector fields $X_{\alpha}$ defined, for small $x > 0$, by 
$$X_{\alpha} (x) = - \frac{x}{1 + \frac{\alpha}{\log(x)}}.$$
The corresponding  flows are shown to be pairwise non bi-Lipschitz conjugate.  Actually, up to composition with a map in the flow, the unique homeomorphism conjugating the flow of $X_{\alpha}$ to the standard linear flow $(t,x)\mapsto e^{-t}\, x$ of $X_0$ is 
$$H_{\alpha} (x) = x \, (| \log(x)|)^{\alpha},$$
which fails to be bi-Lipschitz. This shows that, though Sternberg's linearization holds in the H\"older scale, there is no analogue of it in the $C^{1+ac}$ setting... even for flows! \footnote{As a straightforward computation shows, Sternberg's germ of diffeomorphism (\ref{sternberg-ex}) has absolutely continuous derivative as well; by \cite{EN}, it embeds into a $C^1$ flow (we do not know whether it embeds into a $C^{1+ac}$ flow, but it follows also from \cite{EN} that it is at least $C^1$ conjugate to a diffeomorphism that embeds in such a flow). However, Sternberg proved that it is not bi-Lipschitz conjugate to its linear part.}

\vspace{0.1cm}

Our next result is the analog of Theorem \ref{t:A} for $C^1$ conjugacies.

\begin{thmintro} 
\label{t:B}
Every germ of hyperbolic $C^1$ vector field is contained in a $C^1$-continuous path of such germs, all of them with the same 
multiplier at the origin, and whose time-1 maps of the flows are pairwise bi-Lipschitz conjugate though non $C^1$-conjugate.
\end{thmintro}

Again, our methods are rather concrete and yield explicit examples of vector fields whose flows are non $C^1$ conjugate despite being bi-Lipschitz conjugate. For example, this is the case of 
$$\tilde{X}_{\alpha} (x) = - \frac{x}{1+\frac{\alpha \, cos (\log (|\log(x)|))}{\log (x)}},$$
where $x > 0$ is small. Up to composition with a member of the flow, the only conjugating map to the linear flow is 
$$\tilde{H}_{\alpha} (x) = x \, e^{\alpha \, \sin (\log ( | \log (x)| ))}.$$  
One can readily check that this map is bi-Lipschitz but not $C^1$.

\vspace{0.1cm}

Our last result deals with $C^{1+ac}$ conjugacy classes of germs of hyperbolic diffeomorphisms that are in the same $C^1$ conjugacy class. It corresponds to a weak version of Theorems \ref{t:A} and \ref{t:B} in this setting; see Remark \ref{weak} for possible extensions.

\begin{thmintro} 
\label{t:C}
There exists a $C^{1+ac}$-continuous path of germs of hyperbolic vector fields passing trough the linear vector field \, $X(x) \!=\! -x$ \, such that the time-1 maps of the flows are pairwise $C^1$ conjugate though non $C^{1+ac}$ conjugate.
\end{thmintro}

The results above arose in work of the authors on the path-connectedness of the space of $\mathbb{Z}^d$-actions by $C^{1+ac}$ diffeomorphisms on 1-manifolds. These results show that actions on the interval arising from flows of hyperbolic vector fields cannot be locally managed via a change of coordinates transforming the flow into the linear one. New key ideas are hence needed to deal with this case, and these are introduced in \cite{EN} in a very general framework (that includes the non-hyperbolic setting).

%%%%%%%%%%%%%%%%%%%%%%%%%%%%%%%%%%%%%%%%%%%%%%%%%%%%%%%%%%%%%%%%%%%%%%%

\section{Some general facts on flows and vector fields}

To avoid certain complications on the domains of maps and vector fields, a large part of our discussion will be about the closed half-line $\Rplus := [0,\infty)$. The reader will have no problem in adapting the arguments and results to the setting of germs at the origin (just by extending/modifying maps and vector fields smoothly outside a neighborhood of the origin), so that they fit with those previously stated. 

\vspace{0.1cm}

A flow on $\Rplus$ is a map from $\mathbb{R}\times \Rplus $ to $\Rplus$, 
$$(t,x) \mapsto f^t(x) := f(t,x),$$ 
that satisfies the flow relation
$$f^{t+s} (x) = f^t ( f^s (x))$$ 
for all $x \in \Rplus$ and all $t \in \R$. It is said to be a $C^1$ flow if it is of class $C^1$ on $(t,x)$. In this case, it has an associated (generating) vector field: 
\begin{equation}\label{eq:ODE-gen}
X (x) := \frac{d}{dt} f^t(x) \big|_{t=0}.
\end{equation}
Note that 
$$X (f^t(x)) = \frac{d}{ds} f^s ((f^t) (x)) \Big|_{s=0} =  \frac{d}{ds} f^t ((f^s) (x)) \Big|_{s=0} = Df^t (x) \cdot \frac{d}{ds} f^s (x) \big|_{s=0},$$
hence the vector field is invariant under the action of the flow:
\begin{equation}\label{inv-flow}
X(f^t (x)) = Df^t (x)\cdot X(x). 
\end{equation}

A flow $(f^t)_t$ on $\Rplus$ is said to be {\em (topologically) contracting} if $f^t (x)$ converges to the origin for all $x$ as $t \to +\infty$. This is equivalent to that the {\em time-1 map} $f := f^1$ is contracting, that is, satisfies $f(x) < x$ for all $x \in  \mathbb{R}_+^* := (0,\infty)$. For $C^1$ flows, this is also equivalent to that the associated vector field satisfies $X (x) < 0$ for all $x > 0$. In this case, for all $x > 0$ and $t \in \R$, one has 
\begin{equation}\label{int-flow}
\int_x^{f^t(x)} \frac{ds}{X(s)} = t.
\end{equation}
Indeed, this directly follows by taking derivatives on both sides of this equality and noting that
$$\frac{d}{dt} \int_x^{f^t(x)} \frac{dy}{X(y)} = \frac{df^t(x)}{dt} \cdot \frac{1}{X (f^t(x))} = 1.$$

We define the map $\tau_X : \R_+^* \to \R$ by 
\begin{equation}\label{def:tau-X}
\tau_X (x) := \int_1^x\frac{dy}{X(y)}.
\end{equation}
This map $\tau_X$ is a global (orientation-reversing) $C^1$ diffeomorphism from $\R_+^*$ to $\R$. By (\ref{int-flow}), it satisfies $\tau_X (f^t(1)) = t$ for all $t \in \R$. More generally, for all $x \in \R_+^*$ and $t \in \R$,
$$\tau_X (f^t (x)) = \int_1^{f^t(x)} \frac{1}{X} = \int_x^{f^t(x)} \frac{1}{X} + \int_1^x \frac{1}{X} = t + \tau_X(x).$$
Therefore,
\begin{equation}\label{tau-conj}
\tau_X \circ f^t = T_t \circ  \tau_X
\end{equation}
where $T_t$ denotes the translation by $t$ on the line.

\vspace{0.2cm}

We close this section with a useful lemma on the behavior of contracting vector fields near the origin.

\begin{lem} \label{l:form-vf}
Let $(f^t)_t$ be a contracting $C^1$ flow on $\R_+$ with associated vector field $X$. If $f^1$ is hyperbolic 
at $0$, then there exists $\lambda < 0$ such that, close to the origin, $X (x) \sim \lambda \, x$. 
Moreover, this value $\lambda$ satisfies $Df (0) = e^{\lambda}$.
\end{lem}

\begin{proof} Note that this result is not entirely straightforward because, with our definition of a $C^1$ flow, the generating vector field may be only continuous. The lemma claims that, however, it is differentiable at $0$.

Taking derivatives at the origin in the flow relation $f^{s+t} = f^s \circ f^t$, we obtain 
$$Df^{s+t} (0) = Df^s (0) \cdot Df^t (0).$$
This implies that $Df^s (0)$ must be of the form $e^{\lambda  s}$ for a certain constant $\lambda$. Evaluation at $s=1$ yields $e^{\lambda} = Df (0)$, 
and $\lambda$ must be negative since the flow is assumed to be contracting. Now we compute using (\ref{inv-flow}):
$$ \frac{f(x)-x}{X(x)} 
\,\,\, = \,\,\, \frac{1}{X(x)} \int_0^1 \frac{d}{ds} f^s (x) \, ds 
\,\,\, = \,\,\,  \int_0^1\frac{X(f^s(x))}{X(x)}ds 
\,\,\, = \,\,\,  \int_0^1Df^s(x) \, ds ,$$
hence
\begin{equation} 
\frac{f(x)-x}{X(x)} \,\,\,
\xrightarrow[x \to 0]{} \,\,\, 
\int_0^1Df^s(0) \, ds \,\,\, = \,\,\, \int_0^1 e^{\lambda s} ds \,\,\, = \,\,\,  \frac{ e^{\lambda} - 1}{ \lambda }.
\label{e:equiv}
\end{equation}
Using this and the fact that $f(x) - x \sim x \, (e^{\lambda} - 1)$ for small values of $x$, one easily concludes the proof. 
\end{proof}

%%%%%%%%%%%%%%%%%%%%%%%%%%%%%%%%%%%%%%%%%%%%%%%%%%%%%%%%%%%%%%%%%%%%%%%

\section{Some useful remarks on conjugacies} 

Recall that given homeomorphisms $f,g$ of spaces $M_1,M_2$ respectively, a homeomorphism 
$h \!: M_1 \to M_2$ is said to {\em conjugate} $f$ to $g$ if \, $hfh^{-1} = g$. 

\begin{ex} \label{ex-conj}
Equality (\ref{tau-conj}) says that the map $\tau_X: \R_+^* \to \R$ conjugates the restriction to $\R_+^*$ of a contracting flow $(f^t)_t$ of generating vector field $X$ to the flow of translations on the line.
\end{ex}

The next proposition is crucial; it exhibits a rigidity behavior for conjugacies of diffeomorphisms that arise in a flow. 

\vspace{0.2cm}

\begin{prop} \label{l:conj-time-1}
Let $f,g$ be the time-$1$ maps of two $C^1$ contracting flows $(f^t)_t$ and $(g^t)_t$ of generating vector fields $X$ and $Y$, respectively. Suppose that $h \in\Diff^1(\R_+)$ conjugates $f$ to~$g$. Then it conjugates the flows, that is, $hf^t h^{-1} = g^t$ for all $t \in \mathbb{R}$. Equivalently, it sends $X$ to $Y$, that is $h_*X=Y$, which means that 
\begin{equation}\label{ODE}
Dh=\frac{Y\circ h}{X} \quad \mbox{for all} \quad x \in \R_+^*.
\end{equation}
\end{prop}

\vspace{0.3cm}

For the proof of this proposition, a key role will be played by the next lemma, which can be considered as an extension to $C^1$ diffeomorphisms that arise as time-1 maps of $C^1$ contracting flows of Kopell's famous Lemma. Namely, Kopell proved the statement below for contracting $C^2$ diffeomorphisms, and these are known to belong to a $C^1$ flow by a classical theorem of Szekeres. (Both Kopell's lemma and Szekeres' theorem extend to $C^{1+ac}$ diffeomorphisms, and even to diffeomorphisms with derivative of bounded variation; see \cite{book} and \cite{EN}, respectively. 
However, they do not hold for $C^{1+\alpha}$ diffeomorphisms; see Remark \ref{rem-no-flow} below.) 
We closely follow an argument from \cite{book}.

\vspace{0.3cm}

\begin{lem} \label{l:kopell-gen}
 Let $(f^t)_t$ be a contracting $C^1$ flow with generating vector field $X$. If a $C^1$ diffeomorphism of $\R_+$ commutes with  $f^1$ and has a fixed point in $\R_+^*$, then it equals the identity.
\end{lem}

\begin{proof} Assume $\tilde{f}$ commutes with $f := f^1$ and has a fixed point $x_0\in\R_+^*$. By commutativity, each point of the sequence $(f^n(x_0))$ is fixed by $\tilde{f}$. Note that this sequence converges to the origin. Therefore, $D \tilde{f}(0) = 1$. Denote $I := [f(x_0),x_0]$. The relation $\tilde{f}^k = f^{-n} \tilde{f}^k f^n$ gives, for each $k,n$ in $\mathbb{N}$ and $y \in I$, 
$$D\tilde{f}^k (y) = \frac{Df^n (y)}{Df^n (\tilde{f}^k(y))} \cdot D\tilde{f}^k (f^n (y)).$$
Denoting $y_k := \tilde{f}^k (y)$ and letting $s_k \in [-1,1]$ be such that $f^{s_k} (y) = y_k$, this gives 
$$D\tilde{f}^k (y) = \frac{Df^n (y)}{Df^n (f^{s_k}(y))} \cdot D\tilde{f}^k (f^n (y)) = 
\frac{Df^{s_k} (y)}{Df^{s_k}(f^n(y))}  \cdot D\tilde{f}^k (f^n (y)).$$
Since $D \tilde{f}(0) = 1$, the expression $D\tilde{f}^k (f^n (y))$ converges to $1$ as $n$ goes to infinity. Moreover, 
$$\frac{Df^{s_k} (y)}{Df^{s_k}(f^n(y))} \leq C:= 
\frac{\max \{ Df^s (\bar{y}): s \in [-1,1], \bar{y} \in [0,x_0]\}}{\min \{Df^s (\bar{y}): s \in [-1,1], \bar{y} \in [0,x_0]\}}.$$
We thus conclude that $D\tilde{f}^k$ is uniformly bounded (by $C$) on $I$ independently of $k$. This easily implies that $\tilde{f}$ acts as the identity on $I$. By commutativity, it is the identity on the whole $\R_+^*$. 
\end{proof}

The next lemma is a standard consequence of the previous one.

\begin{lem} \label{l:unique-X}
 Let $(f^t)_t$ be a contracting $C^1$ flow. If a $C^1$ diffeomorphism of $\R_+$  commutes with $ f^1$, then it belongs to the flow. In particular, $(f^t)_t$ is the unique $C^1$ flow whose time-1 map coincides with $f$.
\end{lem}

\begin{proof}
Let $\tilde{f}$ be any element of $\mathrm{Diff}^1 (\R_+)$ that commutes with $f$. Fix $x_0 > 0$, and let $t_0 \in \R$ be such that $f^{t_0} (x_0) = \tilde{f} (x_0)$. Then $ f^{-t_0} \tilde{f}$ fixes the point $x_0$ and commutes with $f$. By the previous lemma, it equals the identity, hence $\tilde{f} = f^{t_0}$. 

To show the second statement, let $(\hat{f}^t)_t$ be a $C^1$ flow for which $\hat{f}^1 = f^1$. By the first part, for each $t \in \R$, there must exist $t' \in \R$ such that $\hat{f}^t = f^{t'}$. The correspondence $t \mapsto t'$ is easily seen to be a continuous group homomorphism of $(\R,+)$ that sends $1$ into $1$. It must hence equal the identity, and therefore $\hat{f}^t = f^t$ for all $t \in \R$. 
\end{proof}

\begin{rem}\label{rem-no-flow}
For each $\alpha < 1$, Tsuboi has built in \cite{tsuboi} examples of $C^{1+\alpha}$ contractions of $\R_+$ whose centralizers contain 
nontrivial elements with infinitely many fixed points. According to  Lemma  \ref{l:kopell-gen}, these contractions do not embed into a $C^1$ 
flow. Another (more elaborate) source of such examples comes from \cite{KN} and works as follows: Fix $0 < \beta < 1 - \alpha$  and build 
two commuting contractions $f, g$ of $\R_+$ of regularities $C^{1+\alpha}$ and $C^{1+\beta}$, respectively, which generate a 
rank-2 Abelian group acting freely but non-minimally on $\R_+^*$. Then $f$ does not embed into a $C^1$ flow. Otherwise, by Lemma 
\ref{l:unique-X} above, $g$ would belong to such flow, and this would imply that the action of the group generated by $f$ and $g$ on 
$\R_+^*$ is minimal, which is supposed not to be the case.
\end{rem}

\begin{rem}\label{rem-no-flow-dos}
More interestingly, there exist examples of hyperbolic $C^1$ contractions that do not embed into a flow. These arise from the work of Farinelli, who 
built $C^1$ hyperbolic contractions that admit nontrivial elements in their centralizer with infinitely many fixed points (see \cite[Theorem 1.7]{farinelli} 
and its proof). Again, by Lemma  \ref{l:kopell-gen}, such a contraction cannot embed into a flow. Note that, unlike the parabolic case 
discussed above, these contractions cannot be of class $C^{1+\alpha}$ because of Sternberg's linearization theorem. 
\end{rem}

\begin{proof}[Proof of Proposition \ref{l:conj-time-1}] 
Let $h$ be a $C^1$ diffeomorphism of $\R_+$ that conjugates $f$ to $g$. Then $(h^{-1} g^t h)_t$ is a $C^1$ flow 
whose time-1 map coincides with $h^{-1} g h = f$. By Lemma \ref{l:unique-X}, we have $h^{-1} g^t h = f^t$ for all $t \in \R$, hence $h$ conjugates the flows. 

Now, if $g^t h (x) = h f^t (x)$ for all $t \in \mathbb{R}$, then taking derivatives at $t = 0$ we obtain
$$Y(h(x)) = \frac{d}{dt} g^t (h(x)) = \frac{d}{dt} h f^t (x) = Dh(x) \frac{d}{dt} f^t (x) = D h(x) \cdot X (x),$$ 
which shows (\ref{ODE}). 

Conversely, if this relation holds, then
$$\frac{d}{dt} (h^{-1} g^t h)(x) 
= Dh^{-1} (g^t h(x)) \, \frac{d}{dt} g^t (h(x))
= \frac{Y (g^t h(x))}{Dh (h^{-1}g^th(x))} = X(h^{-1}g^th(x)).$$
Thus, the flow $(h^{-1} g^t h)_t$ satisfies the same ODE as (\ref{eq:ODE-gen}). If $X$ were of class $C^1$ (or at least Lipschitz), 
the classical uniqueness theorem for the Cauchy problem would conclude the proof. However, $X$ is only continuous, 
and a specific argument is needed. This strongly uses the 1-dimensional setting and the fact that the flows 
involved are contracting, and proceeds as follows: 
Consider the function $\tau_X$ defined by (\ref{def:tau-X}). If $(\tilde{f}^t)_t$ is any other flow that satisfies the ODE  (\ref{eq:ODE-gen}), 
then for all positive $x$ we have 
$$\frac{d \tau_X (\tilde{f}^t(x))}{dt} =D\tau_X(\tilde{f}^t(x))\cdot \frac{d\tilde{f}^t(x)}{dt}=\frac1{X(\tilde{f}^t(x))}\cdot \frac{d\tilde{f}^t(x)}{dt}=1.$$ 
Thus,  
$$\tau_X (\tilde{f}^t(x))=\tau_X(x)+t.$$
This means that $\tau_X$ conjugates $\tilde{f}^t$ to the translation by $t$, which implies the equality $\tilde f^t=f^t$, as required.

\end{proof}

In view of Proposition \ref{l:conj-time-1}, the study of $C^1$ diffeomorphisms that conjugate the time-1 maps of two contracting vector fields transfers into that of those that conjugate the flows. Very nicely, these conjugating maps can be easily characterized even in the framework of homeomorphisms. 

\begin{lem} 
\label{l:uniq}
Let $(f^t)_t, (g^t)_t$ be two contracting flows with generating vector fields $X,Y$, respectively. Suppose that a homeomorphism $h$ of $\R_+^*$ conjugates these flows. Then $h$ is of the form $\tau_Y^{-1}T_t \tau_X$ for some $t \in \R$. Conversely, any map of this form conjugates the flows.
\end{lem}

\begin{proof} 
Consider the associated homeomorphisms $\tau_X$ and $\tau_Y$ from $\R_+^*$ to $\R$. For all $t \in \R$, equality $hf^t = g^th$ yields, in virtue of (\ref{tau-conj}) and the analog property for $\tau_Y$ and $g^t$,
\begin{eqnarray*}
(\tau_Y h \tau_X^{-1}) \,T_t 
&=& (\tau_Y h \tau_X^{-1}) ( \tau_X f^t \tau_X^{-1} )
\,\, = \,\, \tau_Y  (h f^t) \tau_X^{-1} \\
&=& \tau_Y  (g^t h) \tau_X^{-1}  
\,\, = \,\, (\tau_Y g^t \tau_Y^{-1}) (\tau_Y h \tau_X^{-1}) 
\,\, = \,\, T_t \, (\tau_Y h \tau_X^{-1}).
\end{eqnarray*}
Therefore, the homeomorphism $\tau_Y h \tau_X^{-1}$ of $\R$ commutes with all the translations. This easily implies that it is itself a translation, say $\tau_Y h \tau_X^{-1} = T_t$. This shows the announced form $h = \tau_Y^{-1} T_t \tau_X$. The fact that any such map conjugates the flows follows from (\ref{tau-conj}) plus the analogous conjugacy relation $\tau_Y^{-1} g^t \tau_Y = T_t$. 
\end{proof}

\begin{rem} 
In order to illustrate how crucial are the hypothesis in Proposition \ref{l:conj-time-1} and Lemma \ref{l:uniq}, let us point out that, given any contracting $C^1$ flow, it is not hard to build a bi-Lipschitz homeomorphism of $\R_+$ that is not $C^1$ and commutes with the time-1 map $f$ while not centralizing the whole flow. Indeed, if we take {\em any} bi-Lipschitz homeomorphism of a fundamental domain $[f(a),a]$, there is a unique way to extend it so that it commutes with $f$, and it is not hard to see that this extension is still bi-Lipschitz. 
\end{rem}

\vspace{0.1cm}

A direct consequence of the lemma above is that any homeomorphism conjugating the flows as in the statement is actually $C^1$ on $\R_+^*$. The failure of regularity will hence be concentrated only at the origin. To better detect it, we will write the map $x \mapsto \tau_Y^{-1} T_t \tau_X (x)$ in the form 
$$x \mapsto g^{s(x)}(x).$$
The value of $s(x) \in \mathbb{R}$ may be computed as follows: the relation $T_t \tau_X (x) = \tau_Y g^{s(x)} (x)$ yields 
$$t + \int_1^x \frac{1}{X} = \int_1^{g^{s(x)}(x)} \frac{1}{Y} = \int_1^x \frac{1}{Y} + \int_x^{g^{s(x)}(x)} \frac{1}{Y} = \int_1^x \frac{1}{Y} + s(x),$$
hence, 
\begin{equation}\label{rel}
t + \int_1^x \left( \frac{1}{X} - \frac{1}{Y} \right) = s(x)
\end{equation}
The next lemma provides a particularly useful way of thinking on this.

\vspace{0.2cm}

\begin{lem} 
\label{r:form} 
Let $f,g$ be the time-$1$ maps of two $C^1$ contracting flows $(f^t)_t$ and $(g^t)_t$ of generating vector fields $X$ and $Y$, respectively. For every $a \in \R_+^*$, the unique homeomorphism that conjugates $(f^t)_t$ to $(g^t)_t$ and fixes $a$ is $x\mapsto g^{\int_a^x(\frac1X-\frac1Y)}(x)$. This is a $C^1$ diffeomorphism of $\R_+^*$ that extends continuously as a map fixing the origin.
\end{lem}

\begin{proof} 
According to the previous lemma and the discussion above, we need to show that if $t$ is chosen so that $\tau_Y^{-1} T_t \tau_X (a) = a$, then, for all $x \in \R$, 
$$s(x) = \int_a^x \left( \frac{1}{X} - \frac{1}{Y} \right).$$
Now note that $a = \tau_Y^{-1} T_t \tau_X (a) = g^{s(a)} (a)$ implies $s(a) = 0$. By (\ref{rel}), this gives 
$$t = \int_1^a \left( \frac{1}{Y} - \frac{1}{X} \right).$$
Therefore, by (\ref{rel}) again, 
$$s(x) =  \int_1^a \left( \frac{1}{Y} - \frac{1}{X} \right) +  \int_1^x \left( \frac{1}{X} - \frac{1}{Y} \right) 
= \int_a^x \left( \frac{1}{X} - \frac{1}{Y} \right),$$
as desired.
\end{proof}

\vspace{0.1cm}

The lemma above is the ground path for the three main results of this work.
Roughly, the common mechanism of  proof works as follows: The expression
$$ \int_a^x \left( \frac{1}{Y} - \frac{1}{X} \right)$$
captures how much changes the ``time difference'' between $\tau_X (x) - \tau_X (a)$ and $\tau_Y(x) - \tau_Y (a)$. If this difference becomes unbounded, 
then the point $x$ gets pushed more and more in the time scale. For the original flows, this means that the conjugacy between them pushes with exponential 
force, and this translates into a failure of the Lipschitz property. One thus gets Theorem~A. If the difference above is bounded but oscillates, then the conjugacy 
becomes Lipschitz, but not smooth. This is Theorem~B. Finally, the situation for Theorem~C is more complicated (this is why we cannot provide a full 
statement in this case), but the strategy still begins by properly controlling the difference above.

\vspace{0.2cm}

We close this section with a very well-known lemma. Although it is stated for global diffeomorphisms, a 
version for (non-necessarily contracting) germs of diffeomorphisms also holds, as the reader will easily note.

\vspace{0.3cm}

\begin{lem} 
\label{l:pres-mult}
Let $f,g$ be $C^1$ diffeomorphisms of $\R_+$. If they are conjugated by a bi-Lipschitz homeomorphism, then $Df (0) = Dg (0)$.
\end{lem}

\begin{proof} 
We may assume that $f$ and $g$ have no fixed point on a certain interval $(0,\delta]$ (with $\delta > 0$), otherwise they would both have derivative 1 at the origin. Passing to the inverses if necessary, we may suppose that they contract toward the origin on this interval. Then, for all $x \in [0,\delta] $, one has the convergence $Df^k (x)^{1/k} \to Df (0)$. Indeed, 
$$\log \big[ Df^k (x)^{1/k} \big] = \frac{1}{k} \sum_{i=0}^{k-1} \log Df (f^i (x)),$$
and since $f^i(x) \to 0$, the right-side expression converges to $\, \log Df (0)$. With a little extra effort, one checks that the convergence is uniform on $[0,\delta]$. A similar statement holds for $g$.

Denote by $h$ a bi-Lipschitz conjugacy from $f$ to $g$. Fix $x_0 < \delta$, and consider the equality 
$$\frac{h (f^k (x_0)) - h( 0)}{f^k (x_0) - 0} \cdot  \frac{f^k (x_0) - f^k (0)}{x_0 - 0} 
= \frac{h (f^k (x_0))}{x_0} 
= \frac{g^k (h(x_0))}{x_0}
= \frac{g^k (h(x_0)) - g^k (0)}{x_0 - 0}.$$
The expression on the right side equals $Dg^k (\eta_k)$ for a certain $\eta_k \leq h(x_0)$. Up to the bi-Lipschitz constant $C$ of $h$, the expression on the left side equals $Df^k (\zeta_k)$ for a certain $\zeta_k \leq \delta$. This implies that 
$$Dg^k (\eta_k) / C \leq Df^k (\zeta_k) \leq C \cdot Dg^k (\eta_k).$$
If we take $k^{th}$-roots and pass to the limit in these inequalities, the convergences above yields 
$$Dg (0) \leq  Df (0) \leq D g (0),$$
which closes the proof.
\end{proof}

%%%%%%%%%%%%%%%%%%%%%%%%%%%%%%%%%%%%%%%%%%%%%%%%%%%%%%%%%%%%%%%%%%%%%%%%

\section{Non bi-Lipschitz conjugate $C^1$ hyperbolic vector fields} 

We start with a general criterion of bi-Lipschitz conjugacy based on the considerations of the previous sections.

\vspace{0.1cm}

\begin{prop}
\label{crit-conj-bil}
Let $(f^t)_t$ and $(g^t)_t$ be two contracting $C^1$ flows on $\R_+$ of generating vector fields $X$ and $Y$, respectively. Denote by $f,g$ the corresponding time-1 maps, and consider the following properties:

\vspace{0.1cm}

\noindent (i)  $\log \left( \frac{X}{Y} \right)$ and $x\mapsto \int^x_1(\frac1X-\frac1Y)$ are bounded (from above and below) on some interval $(0,\delta]$.

\vspace{0.1cm}

\noindent (ii)  $f$ and $g$ are conjugate by a homeomorphism that is bi-Lipschitz on some interval $(0,\delta]$.

\vspace{0.1cm}

\noindent  Then (i) implies (ii). Moreover, if $f$ is hyperbolic at $0$, then (ii) implies (i).
\end{prop}

\begin{rem}
\label{r:parabolic}
For contracting parabolic flows at the origin, condition (ii) (and even $C^1$ conjugacy) does no longer imply (i). As a concrete example, the flows of the 
vector fields $X(x)=-x^2$ and $Y(x)=-2x^2$ are conjugate by a homothety despite these two vector fields do not satisfy (i). The study of conjugacy 
classes of parabolic vector fields will be the subject of another work \cite{parabolic}.
\end{rem}

\begin{proof}[Proof that (i) implies (ii)] 
Consider the map $h$ defined on $\R_+^*$ by 
$$h(x) := \tau_Y^{-1} \tau_X (x) = g^{\int_1^x(\frac1X-\frac1Y)}(x).$$
By Lemma \ref{r:form}, this is a $C^1$ diffeomorphism of $\R_+^*$ that conjugates $f$ to $g$ and extends to a homeomorphism of $\R_+$ by letting $h(0)=0$. Moreover, by definition of $\tau_X$ and $\tau_Y$ and by \eqref{inv-flow},
\begin{equation}
\label{e:Dh}
Dh(x) = \frac{Y(h(x))}{X(x)} = \frac{Y(x)}{X(x)}\cdot \frac{Y(g^s (x))}{Y(x)} = \frac{Y(x)}{X(x)}\cdot Dg^s (x),
\end{equation}
where $s = \int_1^x ( \frac1X-\frac1Y )$. Suppose that (i) holds, that is, $\left|\log\frac{X}{Y}\right|$ and $x\mapsto \left|\int^x_1(\frac1X-\frac1Y)\right|$ are bounded from above on $(0,\delta]$ by some constant $C$. Then \eqref{e:Dh} implies that $|\log Dh|$ is bounded from above by 
$$C+\max_{(t,x)\in [-C,C]\times(0,\delta]}|\log Dg^t(x)|.$$ 
Therefore, $h$ is bi-Lipschitz on $[0,\delta]$. 
\end{proof}

\begin{proof}[Proof that (ii) implies (i) in the hyperbolic case]
Conversely,  assume that there exists a homeomorphism $h$ conjugating $f$ to $g$ that is bi-Lipschitz on a neighborhood of the origin. Moreover, assume that $Df(0)=e^\lambda < 1$. 
By Lemma \ref{l:pres-mult}, we also have $Dg (0) = e^{\lambda}$. Now write 
$$
\log\left(\frac X Y\right) 
= \log\left(\frac{X}{f-\id}\right) + \log\left(\frac{f-\id}{g-\id}\right)+\log\left(\frac{g-\id}{Y}\right). 
$$
Since $f-\id\sim(e^\lambda-1) \, \id$ and $g-\id\sim(e^\lambda-1) \, \id$ close to the origin, the central term above converges to 0 at $0$. Moreover, by (\ref{e:equiv}), 
$$\frac{f-\id}{X} \longrightarrow \frac{e^{\lambda} - 1}{\lambda} 
\qquad \mbox{and} \qquad 
\frac{g-\id}{Y} \longrightarrow \frac{e^{\lambda}-1}{\lambda}$$
at $0$,  so the first and last terms above cancel each other in the limit. We thus conclude that $X(x) / Y(x) \to 1$ as $x \to 0$ and, in particular, $\log(X/Y)$ is bounded on $(0,\delta]$.
 
Note that $hf^t$ still conjugates $f$ to $g$ for all $t \in \R$. By adjusting $t$, we may assume that $h(1)=1$. Given a positive $x$ close to $0$, let $k \in \mathbb{N}$ be such that $f^{k+1} (1) < x \leq f^k (1)$. Then 
 \begin{eqnarray*}
  \int_1^x \left( \frac1X-\frac1Y \right) 
& =&   \int^{f^k(1)}_{1} \left( \frac1X-\frac1Y \right) + \int^x_{f^k(1)} \left( \frac1X-\frac1Y \right) \\
& = & \int^{f^k(1)}_{1} \frac1X - \int^{g^k(1)}_{1} \frac1Y 
 + \int^{g^k(1)}_{f^k(1)} \frac1Y  + \int^x_{f^k(1)} \left( \frac1X-\frac1Y \right) \\
 &=& \int^{g^k(1)}_{f^k(1)} \frac1Y  + \int^x_{f^k(1)} \left( \frac1X-\frac1Y \right) ,
 \end{eqnarray*}
 where the last equality follows from 
 $$\int_1^{f^k(1)} \frac{1}{X} = k = \int_1^{g^k(1)} \frac{1}{Y}.$$ 
Now writing $Y(x) = \lambda \, x \, (1+u(x))$ for some continuous function $u$ going to $0$ as $x \to 0$ as in Lemma \ref{l:form-vf}, we obtain 
\begin{equation}
\label{e:truc2}
\int^{g^k(1)}_{f^k(1)}\frac1Y 
= \frac1\lambda\int^{g^k(1)}_{f^k(1)}\frac{dy}{y(1+u(y))} 
= \frac{1}{\lambda} \int^{\log(g^k(1))}_{\log(f^k(1))}\frac{dz}{1+u(e^z)}=\frac{1}{|\lambda|}\log\left(\frac{f^k(1)}{g^k(1)}\right) \frac1{1+u(e^{z_k})},
\end{equation}
for some $z_k$ between $f^k(1)$ and $g^k(1)$. Now for $k$ big enough, $|u(y)|\le\frac12$ for every $y$ between $f^k(1)$ and $g^k(1)$, so 
\begin{equation*}
\left|\int^{g^k(1)}_{f^k(1)}\frac1Y\right|\le \frac2{|\lambda|}\left|\log\left(\frac{f^k(1)}{g^k(1)}\right)\right| = \frac2{|\lambda|}\left|\log\left(\frac{f^k(1)-0}{h(f^k(1))-h(0)}\right)\right|,
\end{equation*}
which is bounded since $h$ is bi-Lipschitz. Finally, 
\begin{equation}\label{e:truc-n}
\left|  \int^x_{f^k(1)} \left( \frac1X-\frac1Y \right)  \right| \leq 
\left|\int_{f^k(1)}^{f^{k+1}(1)}\left|\frac1X-\frac1Y\right|\right| 
= \left|\int_{f^k(1)}^{f^{k+1}(1)}\left|\frac1X\right|\cdot\left|1-\frac{X}Y\right|\right|
\le C\underbrace{\left|\int_{f^k(1)}^{f^{k+1}(1)}\left|\frac1X\right|\right|}_{1},
\end{equation}
where $C$ is an upper bound for  $\left|1-\frac{X}Y\right|$ (which exists since we already proved that $\log (\frac{X}{Y})$ is bounded). Putting all of this together, we conclude that 
$$\left| \int_1^x \left( \frac{1}{X} - \frac1Y \right)  \right|$$
is bounded on a certain interval $(0,\delta]$, as announced.
\end{proof}

The preceding proposition allows building a lot of examples of $C^1$ hyperbolic flows that are not pairwise bi-Lipschitz conjugate. For example, start with the linear vector field $X(x) = -x$ and consider a $C^1$ perturbation $X_{\alpha}$ that, for small values of $x$, satisfies
$$X_{\alpha}(x) := - \frac{x}{1 + \frac{\alpha}{\log (x)}}.$$
One can easily check that $X_{\alpha}$ is of class $C^1$ on $\R_+$. (Actually, $X_{\alpha}$ can be taken to be of class $C^{1+ac}$ on compact subsets of $\mathbb{R}_+$;  we will come back to this point in \S \ref{sec-var}.) Moreover, its multiplier at the origin equals $-1$ for all $\alpha$. 
Furthermore, though $\big| \log (\frac{X_{\alpha}}{X_{\beta}}) \big|$ remains bounded for $\alpha \neq \beta$ (actually, it converges to 0 as $x \to 0$), the expression $ \big| \int^x_1(\frac{1}{X_{\alpha}}-\frac{1}{X_{\beta}}) \big|$ behaves as 
$$\left| (\alpha - \beta ) \int^x_1 \frac{dy}{y \, \log (y)} \right|$$
which is not bounded near $0$. Proposition \ref{crit-conj-bil} then shows that the corresponding flows $(f^t_{\alpha})_t$ and $(f^t_{\beta})_t$ are not bi-Lipschitz conjugate for $\alpha \neq \beta$. In particular, none of the flows $(f^t_{\alpha})_t$ is bi-Lipschitz conjugate to the affine flow for $\alpha \neq 0$.

We can use previous discussions to obtain explicit conjugating maps and analyse the failure of the bi-Lipschitz property. Indeed, by Lemmas \ref{l:uniq} and \ref{r:form}, any map $h$ that conjugates $(f^t_{\alpha})_t$ to the linear flow $(f^t)_t$ (with vector field $X = X_0$) is given by $x \mapsto f^{s(x)}(x)=e^{-s(x)}x$, where, for small values of $a$ and $x$ and for some constant $c$,
$$s (x) = c + \int_a^x \left( \frac{1}{X_{\alpha}} - \frac{1}{X} \right)  = c - \alpha \int_a^x  \frac{dy}{y \, \log (y)} = 
c' + \log \big( (| \log(x) | )^{-\alpha} \big).
$$ 
This way, a very nice conjugating map arises, namely 
$$H_{\alpha} (x) = x \, \big( | \log(x) | \big)^{\alpha} $$ 
(where, again, equality holds for small $x$), and all the others are compositions of it with a homothety.

Of course, similar tricks produce $C^1$ contracting hyperbolic vector fields on $\R_+$ lying in many other (local) bi-Lipschitz conjugacy classes. For example, straightforward computations based on Proposition \ref{crit-conj-bil} show that the flow of $X_{\alpha}$ is in the same (local) bi-Lipschitz conjugacy class as the flow of 
 \footnote{
This example for $\alpha = 1$ was already considered in \cite[Exercise 4.1.12]{book}. 
However, our approach here is much more concrete (and is not based only on a ``miraculous'' direct integration). }
$$Y_{\alpha} (x) := - x \, \left( 1 - \frac{\alpha}{\log(x)} \right),$$
but (for $\alpha \neq 0$) this does not coincide with that of the flow of 
$$\bar{X}_{\alpha} (x) := - \frac{x}{1 + \frac{\alpha}{\log(x) \, \log ( | \log(x) | ) }}.$$ 
Besides, this is different from that of the flow of 
$$\bar{\bar{X}}_{\alpha} (x) := - \frac{x}{1+\frac{\alpha}{\log(x) \, \log ( | \log(x) |) \, \log ( \log ( | \log(x) | ) ) }},$$
etc. The idea behind the first example is, however, enough to prove Theorem \ref{t:A}.

\begin{proof}[Proof of Theorem \ref{t:A}] 
Let $X$ be a hyperbolically contracting $C^1$ vector field on $\R_+$.  Given $\alpha \neq 0$, consider a vector field $X_{\alpha}$ that, close to the origin, satisfies 
$$X_{\alpha} := \frac{X}{1 + \frac{\alpha X}{x \, \log(x)}}.$$
One can compute:
$$DX_{\alpha} 
= \frac{1}{\left( 1 + \frac{\alpha X}{x \log(x)}\right)^2} \left[ DX \, \left( 1 + \frac{\alpha \, X}{x \, \log(x)} \right) 
- \alpha \, \Big( \frac{X}{x} \Big) \cdot \frac{DX}{\log(x)} + \alpha \, \Big( \frac{X}{x} \Big)^2 \cdot \frac{( \log(x)+1)}{(\log(x))^2} \right].$$
The expression above is continuous on $(0,\delta]$ for small $\delta > 0$. Moreover, since there is $\lambda < 0$ such that $X(x) \sim \lambda x$ when $x \to 0$ (see Lemma \ref{l:form-vf}), it converges to $DX (0) = \lambda$ as $x \to 0$. One can then easily show that the $X_{\alpha}$ are $C^1$ and converge to $X$ in $C^1$ topology as $\alpha \to 0$ (on some interval $[0,\delta]$...). 

Now we compute for small $a,x$:
$$\int_a^x \frac{1}{X_{\alpha}} - \frac{1}{X_{\beta}} = ( \alpha - \beta) \int_a^x \frac{dy}{y \, \log (y)}.$$
Since the last integral diverges as $x \to 0$ when $\alpha\neq \beta$, Proposition \ref{crit-conj-bil} implies that the flows of $X_{\alpha}$ and $X_{\beta}$ are not bi-Lipschitz conjugate. 
\end{proof}

%%%%%%%%%%%%%%%%%%%%%%%%%%%%%%%%%%%%%%%%%%%%%%%%%%%%%%%%%%%%%%%%%%%%%%%

\section{Bi-Lipschitz conjugate, non $C^{1}$ conjugate germs} 

The proof of Theorem \ref{t:B} follows a similar strategy to that of Theorem \ref{t:A}. To carry it out, we first modify the criterion of Proposition \ref{crit-conj-bil} to a criterion of $C^1$ conjugacy.

\begin{prop} 
\label{crit-conj-C1}
Let $(f^t)_t$ and $(g^t)_t$ be two contracting $C^1$ flows on $\R_+$ of generating vector fields $X$ and $Y$, respectively. Denote by $f$ and $g$ the corresponding time-1 maps, and consider the following properties:

\vspace{0.1cm}

\noindent (i)  $\log \left( \frac{X}{Y}(x) \right)$ and $\int^x_1(\frac1X-\frac1Y)$ converge as $x \to 0$.

\vspace{0.1cm}

\noindent (ii) $f$ and $g$ are conjugate by a $C^1$ diffeomorphism.

\vspace{0.1cm}

\noindent  Then (i) implies (ii). Moreover, if $f$ is hyperbolic at $0$, then (ii) implies (i).
\end{prop}

\begin{proof} 
To see that (i) implies (ii), recall relation (\ref{e:Dh}) for the conjugating map $h$: 
$$Dh(x) = \frac{Y(h(x))}{X(x)} =  \frac{Y(x)}{X(x)}\cdot Dg^s (x),$$
where $s =  \int^x_1(\frac1X-\frac1Y)$.  If both $\log (X/Y)$ and $x\mapsto \int^x_1(\frac1X-\frac1Y)$ converge at the origin, this implies that $Dh$ extends continuously to the origin by $Dh (0) \neq0$. Since $h$ is a $C^1$ diffeomorphism of $\R_+^*$, this shows that $f$ and $g$ are $C^1$ conjugate.

Conversely, assume that $f$ is hyperbolic at $0$ and that $f$ and $g$ are $C^1$ conjugate by some diffeomorphism $h$, which again can be assumed to satisfy $h(1)=1$ (by composing it with some $f^t$ if necessary). Then by Proposition \ref{l:conj-time-1}, $hf^t(1)=g^th(1)=g^t(1)$ for every $t\in\R$. Now if we write again $Y(x) = \lambda \, x \, (1+u(x))$ for some continuous function $u$ going to $0$ as $x \to 0$ (\emph{cf.} Lemma \ref{l:form-vf}), for every $x=f^t(1)>0$, we get, similarly as in the proof of Proposition \ref{crit-conj-bil},
\begin{align*}
\int_1^x \left(\frac1X-\frac1Y\right)
& = \underbrace{\int_1^{f^t(1)} \frac1X-\int_1^{g^t(1)}\frac1Y}_{t-t\,=\,0} + \int_{f^t(1)}^{g^t(1)}\frac1Y=\int_x^{h(x)}\frac1Y\\
&= \frac1\lambda\int^{h(x)}_{x}\frac{dy}{y(1+u(y))} 
= \frac{1}{\lambda} \int^{\log(h(x))}_{\log(x)}\frac{dz}{1+u(e^z)}
= \frac{1}{\lambda}\log\left(\frac{h(x)}{x}\right) \frac1{1+u(y_x)}
\end{align*}
for some $y_x$ between $x$ and $h(x)$, so this quantity converges to $\lambda^{-1}\log Dh(0)$, which completes the proof.
\end{proof}

As in the previous section, the proposition above allows building many examples of hyperbolic vector fields with the same multipliers at the origin whose flows are bi-Lipschitz conjugate yet not $C^1$ conjugate. For example, start with $X(x) = -x$ and consider a $C^1$ perturbation $\tilde{X}_{\alpha}$ that, for small values of $x$, satisfies
$$\tilde{X}_{\alpha}(x) := - \frac{x}{1 + \frac{\alpha \, \cos(\log(|\log(x)|))}{\log (x)}}.$$
Again, one readily checks that $\tilde{X}_{\alpha}$ is of class $C^1$ on $\R_+$ (actually it can be taken of class~$C^{1+ac}$ on compact subsets of $\mathbb{R}_+$). Moreover, while the value of $\log (\tilde{X}_{\alpha} / \tilde{X}_{\beta})$ goes to $0$ as $x \to 0$, for small $a,x$ we have 
$$\int_a^x \left( \frac{1}{\tilde{X}_{\alpha}} - \frac{1}{\tilde{X}_{\beta}} \right) 
= \int_a^x \frac{(\beta - \alpha) \, \cos (\log ( | \log (y) | ))}{y \, \log (y)} dy 
= (\beta - \alpha) \, \sin (\log (| \log (x) | )) + c.$$
Since this expression remains bounded but does not converge as $x \to 0$ if $\alpha\neq \beta$, there exists a locally bi-Lipschitz conjugacy from the flow of $\tilde{X}_{\alpha}$ to that of $\tilde{X}_{\beta}$, but no $C^1$ conjugacy. Actually, a conjugacy can be made explicit, namely, $\tilde{H}_{\alpha,\beta} = \tilde{H}_{\beta}^{-1} \tilde{H}_{\alpha}$, where 
$$\tilde{H}_{\alpha} (x) = x \, e^{\alpha \, \sin (\log (| \log (x)|))}.$$  

\begin{proof}[Proof of Theorem \ref{t:B}] 
We proceed similarly as for Theorem \ref{t:A}: starting with a hyperbolically contracting $C^1$ vector field $X$ on $\R_+$ and  $\alpha \neq 0$, consider a vector field $\tilde{X}_{\alpha}$ that, close to the origin, satisfies 
$$\tilde{X}_{\alpha} := - \frac{X}{1 + \frac{\alpha X \, \cos (\log ( | \log (x) | ))}{x \, \log(x)}}.$$
Checking the details is left to the reader. 
\end{proof}

%%%%%%%%%%%%%%%%%%%%%%%%%%%%%%%%%%%%%%%%%%%%%%%%%%%%%%%%%%%%%%%%%%%%%%%

\section{$C^1$ conjugate, non $C^{1+ac}$ conjugate diffeomorphisms} 
\label{sec-var}

It seems hard to obtain a general criterion of either $C^{1+bv}$ or $C^{1+ac}$ conjugacy. Nevertheless, we can somehow mimic the constructions from the previous sections in order to obtain hyperbolic vector fields that are $C^1$ conjugate yet not $C^{1+ac}$ conjugate. Indeed, we can even start with an appropriate candidate of conjugacy, and reverse the steps of the previous sections.

Let us begin by considering the conjugating map $h (x) = g^{s(x)} (x)$, where $(g^t)_t$ is the linear flow (associated to the vector field $X(x) = -x$). This means that $h(x) = e^{-s(x)} x$. Assume that this map conjugates the flow of a vector field
$$X(x) := - \frac{x}{1+u(x)}$$
to the linear flow and fixes $a \in \mathbb{R}_+^*$. Then, according to Lemma \ref{r:form}, 
$$s(x) = \int_a^x \left( \frac{1}{X(y)} - \left( - \frac{1}{y} \right) \right) \, dy = \int_a^x - \frac{u(y)}{y},$$
hence $u = - x \, Ds$. Note that $h, u$ and $X$ are uniquely determined by $s$. We summarize 
all of this below:
\begin{equation}\label{l:inv}
h(x) = e^{-s(x)} x, \qquad  u(x) = - x \, Ds (x), \qquad X(x) = - \frac{x}{1+u(x)}.
\end{equation}

\begin{lem} 
\label{l:conditions-ac} 
Let $s \!: \R_+ \to \R$ be a function which is smooth (at least $C^3$) on $\R_+^*$ and satisfies $s(0)=0$, and let $u :=-x\, Ds$. Assume  that the following conditions are satisfied for a certain $\delta > 0$:

\vspace{0.1cm}

\noindent {\bf (i)} the function $s$ is continuous; 

\vspace{0.1cm}

\noindent {\bf (ii)} the function $u$ is continuous, with $u(0) = 0$ and $u (x) \geq c > -1$ for some constant $c$ 
and for all $x$;

\vspace{0.1cm}

\noindent {\bf (iii)} the derivative $Du$ restricted to $(0,\delta]$ belongs to $L^1$;

\vspace{0.1cm}

\noindent {\bf (iv)} the restriction of $Ds$ to $(0,\delta]$ does not belong to $L^1$; 

\vspace{0.1cm}

\noindent {\bf (v)} the function $x \, Du$ extends continuously to the origin with value $0$ at this point; 

\vspace{0.1cm}

\noindent {\bf (vi)} the function $x \, (Du)^2$ belongs to $L^1$ when restricted to $(0,\delta]$;

\vspace{0.1cm}

\noindent {\bf (vii)} the function $D(x \, Du)$ belongs to $L^1$ on $(0,\delta]$.

\vspace{0.1cm}

\noindent Then the map $h$ defined as in \eqref{l:inv} is a $C^1$ diffeomorphism whose derivative has unbounded variation, but it conjugates the linear flow $(t,x) \mapsto e^{-t}x$ to the flow of a $C^{1+ac}$ vector field $X$ (namely, the one defined as in \eqref{l:inv}).
\end{lem}

\begin{proof} 
All regularity issues are concentrated near the origin, so we work on $[0,\delta]$. We first compute
$$Dh = e^{-s} + x \cdot (-Ds)\,  e^{-s} = e^{-s} \, (1 - x\, Ds) = e^{-s} \, (1+u),$$
which by (i) and (ii) shows that $h$ is a $C^1$ diffeomorphism. For the second derivative, one easily gets 
$$D^2 h = e^{-s} \, \big( Du - (1+u) Ds \big).$$
Since $s$ and $u$ are continuous (hence bounded on $[0,\delta]$), from (iii) and (iv) we deduce that $D^2 h$ 
does not belong to $L^1 ([0,\delta])$. As a consequence, $Dh$ has unbounded variation, and therefore $h$ 
is not a $C^{1+ac}$ diffeomorphism. 

As for the vector field $X$, we compute:
$$DX = - \frac{1}{1+u} + \frac{x \, Du}{(1+u)^2}.$$
By (ii) and (v), $X$ is of class $C^1$. For the second derivative, a straightforward computation gives 
$$D^2 X = \frac{Du}{(1+u)^2} + \frac{D (x \, Du)}{(1+u)^2} - \frac{2x (Du)^2}{(1+u)^3}.$$
By (ii), (iii), (vi) and (vii), $X$ is of class $C^{1+ac}$. 
\end{proof}

\vspace{0.35cm}

As a concrete example, take a smooth function that, for small $x>0$, satisfies 
\begin{equation}\label{e:ese}
s(x) := - \frac{\sin (\log (|\log (x)|))}{\log (|\log (x)|)}.
\end{equation}
This is a reparametrization of the famous function $x \mapsto x\, \sin (1/x)$, 
which is very well-known to be continuous but with unbounded variation on 
$(0,\delta]$ for, say, $\delta = e^{-10}$. Thus, $Ds$ does not belong to $L^1 ([0,\delta])$, and 
conditions (i) and (iv) of Lemma \ref{l:conditions-ac} are satisfied. 
Most of the remaining properties will follow from the fact that the change of 
variable $x \mapsto 1 / \log(|\log(x)|)$ is a very slow changing map. Concretely, we compute:
\begin{equation}\label{De-u}
u = -x\, Ds(x) 
= \frac{\sin (\log (|\log(x)|))}{|\log(x)| \, ( \log ( |\log(x) |) )^2} -  \frac{\cos (\log (|\log(x)|))}{|\log(x)| \, \log ( |\log(x) |)}.
\end{equation}
This clearly satisfies (ii) on a small interval containing the origin. For the derivative of $u$ one readily gets that $\, Du (x) \,$ equals  
\begin{eqnarray*}
&&
\frac{1}{x \, |\log(x)| \, \log (| \log (x) |)} 
\Big[ 
 \frac{\cos (\log (|\log(x)|))}{ |\log(x)| }  + 
 \frac{2 \, \cos (\log (|\log(x)|))}{|\log(x)|\, \log (|\log(x)|)}  \\
&-&  \frac{\sin (\log (|\log(x)|))}{ |\log(x)| \, \log (|\log(x)|)}  - 
 \frac{2 \, \sin (\log (|\log(x)|))}{|\log(x)|\, (\log (|\log(x)|))^2}  + 
  \frac{\sin (\log (|\log(x)|))}{ |\log(x)| } \Big]
\end{eqnarray*}
This may be written in a resumed form as follows: 
\begin{equation}\label{e:gen}
\frac{1}{x \, |\log(x)| \, \log (| \log (x) |)} \sum \frac{T (\log (|\log(x)|))}{|\log(x)|^{\alpha} \, (\log (|\log (x)|) )^{\beta}},
\end{equation}
where $T$ stands for a trigonometric function (either $\sin$ or $\cos$) and the exponents $\alpha$ and $\beta$ are nonnegative integers of which at least one is positive. Since the functions 
$$\frac{1}{x \, |\log(x)|^2 \, \log (| \log (x) |)} \qquad \mbox{ and } \qquad \frac{1}{x \, |\log(x)| \, (\log (| \log (x) |))^2} $$ 
are integrable on $[0,\delta]$, we have that $Du$ is in $L^1([0,\delta])$, which is condition (iii). Conditions (v) and (vi) also easily follow from this expression. Finally, in order to check condition (vii), we compute $D(x\, Du)$ by multiplying expression (\ref{e:gen}) by $x$ and then deriving. We obtain 
$$\sum \frac{DT (\log (|\log(x)|))}{x \, |\log(x)|^{\alpha + 2} \, (\log (|\log(x)|))^{\beta+1}} + $$
$$+ \sum \frac{T (\log (|\log(x)|))}{|\log(x)|^{2{\alpha+2}} \, (\log (|\log (x)|) )^{2{\beta+2}}} 
\left[ \frac{(\alpha + 1) \, |\log(x)|^{\alpha} (\log (|\log(x)|))^{\beta+1}}{x} \right] + $$
$$+ \sum \frac{T (\log (|\log(x)|))}{|\log(x)|^{2{\alpha+2}} \, (\log (|\log (x)|) )^{2{\beta+2}}} 
\left[ \frac{(\beta + 1) \, |\log(x)|^{\alpha+1} (\log (|\log(x)|))^{\beta}}{x \, |\log(x)|} \right],$$
and this expression belongs to $L^1([0,\delta])$. 

\vspace{0.15cm}

Observe that the flow $f^t := h^{-1} g^{t} h$ of the thus $C^{1+ac}$ vector field $X$ is made of $C^{1+ac}$ diffeomorphisms. This can be checked via direct computations, but we prefer to see it as a consequence of the next general proposition.

\vspace{0.2cm}

\begin{prop} 
\label{p:flow-ac}
The flow of a $C^{1+ac}$ contracting vector field on $[0,1)$ is a 1-parameter group of locally $C^{1+ac}$ diffeomorphisms.
\end{prop}

\begin{proof}
It is standard that the flow $(f^t)_t$ is a 1-parameter group of $C^1$ diffeomorphisms. Moreover, 
the invariance of $X$ under its flow gives $Df^t = \frac{X\circ f^t}{X}$ on $(0,1)$, so $f^t$ is $C^2$ therein. Therefore, for every $0 < a < 1$,
$$\var(\log Df^t; [0,a]) = \int_0^a|D\log Df^t|.$$
To prove that $f^t$ is a locally $C^{1+ac}$ diffeomorphism we hence need to show that the value of this integral is finite. Now using twice the equality $Df^t = \frac{X\circ f^t}{X}$, we obtain
\begin{eqnarray*}
D\log Df^t 
&=& D\log(X\circ f^t)-D\log X \\
&=& (D\log X)\circ f^t \cdot Df^t - D\log X \\
&=& \tfrac{DX\circ f^t\cdot Df^t}{X\circ f^t}-\tfrac{DX}X \\
&=& \tfrac{DX\circ f^t - DX}{X}.
\end{eqnarray*}
Thus, for $t > 0$,
\begin{align*}
\var(\log Df^t; [0,a]) 
&=  \int_0^a\left|\tfrac{DX\circ f^t (x) - DX (x)}{X (x)}\right| dx \\
&= \sum_{n\geq 0}  \int_{f^{(n+1)t}(a)}^{f^{nt}(a)}\left|\tfrac{DX \circ f^t (x) - DX (x)}{X (x)}\right| dx \\
&\le  \sum_{n\geq 0} \max_{[f^{(n+1)t}(a),f^{nt}(a)]}\left|DX\circ f^t - DX\right| \underbrace{\int_{f^{(n+1)t}(a)}^{f^{nt}(a)}\left|\tfrac{1}{X (x)} \right| dx}_{t}\\
&\le \textcolor{red}{2} \, t\cdot\var(DX;[0,a]),
\end{align*}
and a similar estimate holds for $t<0$ changing $t$ by $|t|$ (and the interval $[0,a]$ by $[0,f^{-t}(a)]$ in the last expression).  

Note moreover that this upper bound converges to $0$ as $|t|$ goes to $0$. This easily implies that $t\mapsto\log Df^t$ is continuous in the $C^{ac}$ topology, which completes the proof.
\end{proof}

\vspace{0.2cm}

\begin{rem}
Thanks to the control on $\var\log Df^t$ in the proof above, Proposition \ref{p:flow-ac} easily extends to complete $C^{1+ac}$ vector fields defined on  intervals (with no restriction on the vanishing set). 
\end{rem}

\vspace{0.3cm}

In order to prove Theorem \ref{t:C}, we next move to a parametric version of the previous example. Fix $\alpha \in \mathbb{R}$ and consider the function $s_{\alpha} (x) := \alpha \, s(x)$, where  $s$ is defined as in (\ref{e:ese}). Let $u_{\alpha}(x) := - x \, Ds_{\alpha} (x),$ and let $X_{\alpha}$ be the vector field defined by 
$$X_{\alpha}(x) := - \frac{x}{1+u_{\alpha}(x)}.$$
The corresponding flow, denoted $(f_{\alpha}^t)_t$, is conjugate to the standard linear flow $(g^t)_t$ via the map  $h_{\alpha}(x) = e^{-s_{\alpha}(x)} x.$  The estimate given in the next lemma will be crucial below.

\vspace{0.1cm}

\begin{lem}
\label{l:estimate-D2D}
For a fixed small $\delta > 0$, as $\varepsilon$ goes to 0 we have the asymptotic estimate
$$\int_{\varepsilon}^{\delta} | Ds | \, \, \sim \, \, \frac{ 2 }{\pi} \, \log (\log ( | \log (\varepsilon)| ) ).$$
\end{lem}

\begin{proof} One can proceed via direct computations, but we prefer to work with the auxiliary function $\hat{s} (x) := - x \, \sin (1/x)$. 
Indeed, the relation \, $s (x) = \hat{s}( 1 / \log (|\log(x)|))$ \, gives 
\begin{equation}\label{eq-change} 
\int_{\varepsilon}^{\delta} | Ds | 
= \mathrm{var} (s; [\varepsilon,\delta]) 
= \mathrm{var} (\hat{s} ; [1 / \log(|\log (\varepsilon)|) , 1 /  \log (|\log(\delta)|) ]). 
%= \int_{ \log(|\log (\varepsilon)|)}^{ \log (|\log(\delta)|)} | D \hat{s} | .
\end{equation}
Note that the derivative 
$$D \hat{s} (x) = \frac{1}{x} \cos \left(\frac{1}{x} \right) - \sin \left( \frac{1}{x} \right) $$
vanishes at points $x$ at which 
$$\tan \big( \frac{1}{x} \big) = \frac{1}{x}.$$ 
Now the points $a \geq 0$ satisfying $\tan (a) = a$ form an increasing sequence $(a_k)$ such that $a_k \in [k \pi - \pi/2, k \pi + \pi / 2]$. 
Thus, the function $\hat{s}$ is monotone on each interval $[1/a_{k+1},1/a_k]$, and vanishes at each point $x = 1 / k \pi$. 
%where $\log(|\log(b_k)|) = a_k$, and has different signs at the endpoints. 
From this, one easily concludes that 
\begin{equation}\label{eq-otra}
\mathrm{var} (\hat{s} ; [1 / \log(|\log (\varepsilon)|) , 1 /  \log (|\log(\delta)|) ])
 \, \sim \,  \sum_{k=1}^K | \hat{s} (1 / a_k) | \!+\! | \hat{s} (1 / a_{k+1})| 
\, \sim \, 2 \sum_{k=1}^K | \hat{s} (1/a_k)|,
\end{equation}
where $K \sim \log ( |\log (\varepsilon)|)$ \, (this last estimate follows from that 
$a_K \sim \log (|\log (\varepsilon)|)$). Now the equality $\tan (a_k) = a_k$, gives 
$$| \hat{s} (1 / a_k)| = \left| \frac{\sin (a_k) }{a_k} \right| = | \cos (a_k) | = \frac{1}{\sqrt{1+\tan^2(a_k)}} = \frac{1}{\sqrt{1+a_k^2}},$$
from which we conclude that  
$$2 \sum_{k=1}^K | \hat{s} (1/a_k)| \,\,
= \, \, 2 \sum_{k=1}^K \frac{1}{\sqrt{1+a_k^2}} \,\, 
\sim \,\, 2 \sum_{k=1}^K \frac{1}{a_k} \,\, 
\sim \,\, 2 \sum_{k=1}^K \frac{1}{k \pi } \,\, 
\sim \,\, \frac{2}{\pi} \log (K).$$
Since $K \sim \log ( |\log (\varepsilon)|)$, an application of (\ref{eq-change}) and (\ref{eq-otra}) closes the proof. 
%Now the points $a \geq 0$ satisfying $\tan (a) = a$ form an increasing sequence $(a_k)$ such that 
%$a_k \in [k \pi - \pi/2, k \pi + \pi / 2]$. Thus, the function $s$ is monotone on each interval 
%$[b_{k+1},b_k]$, where $\log(|\log(b_k)|) = a_k$, and has different signs at the endpoints. From this, one easily concludes that 
%$$\int_{\varepsilon}^{\delta} | Ds | \, \, \sim \, \, \sum_{k=1}^K | s (b_k) | + | s (b_{k+1})| 
%\,\, \sim \,\, 2 \sum_{k=1}^K |s(b_k)|,$$
%where $K \sim \log ( |\log (\varepsilon)|)$. Now since $\tan (a_k) = a_k$, we have 
%$$|s (b_k)| = \left| \frac{\sin (a_k)}{a_k} \right| = | \cos (a_k) | = \frac{1}{\sqrt{1+\tan^2(a_k)}} = \frac{1}{\sqrt{1+a_k^2}},$$
%from which we conclude that  
%$$\int_{\varepsilon}^{\delta} | Ds | \, \, 
%\sim \, \, 2 \sum_{k=1}^K \frac{1}{\sqrt{1+a_k^2}} \,\, 
%\sim \,\, 2 \sum_{k=1}^K \frac{1}{a_k} \,\, 
%\sim \,\, 2 \sum_{k=1}^K \frac{1}{k \pi } \,\, 
%\sim \,\, \frac{2}{\pi} \log (K).$$
%Since $K \sim \log ( |\log (\varepsilon)|)$, this closes the proof.
\end{proof}

\vspace{0.35cm}

Now consider two parameters $\alpha$ and $\beta$. The conjugacy relations with respect to the linear flow $(g^t)_t$, namely \esp $h_{\alpha} f^t_{\alpha} h_{\alpha} ^{-1} = g^t = h_{\beta} f^t_{\beta} h_{\beta}^{-1},$ \esp imply 
$$H f_{\alpha}^t H^{-1} = f_{\beta}^t,$$
where $H = H_{\alpha,\beta} := h_{\beta}^{-1} h_{\alpha}$. Up to composition with a member of the flow $(f^t_{\alpha})_t$, this conjugacy map $H$ is unique. The next proposition implies that, for $\alpha \neq \beta$, there is no $C^{1+ac}$ conjugacy between $(f^t_{\alpha})_t$ and $(f^t_{\beta})_t$, and thus closes the proof of Theorem \ref{t:C}.

\vspace{0.2cm}

\begin{lem} 
For $\alpha \neq \beta$, the derivative of the diffeomorphism $H = H_{\alpha,\beta}$ 
has unbounded variation on any interval containing the origin.
\end{lem}

\begin{proof} 
Since $H$ is a $C^1$ diffeomorphism that is $C^2$ outside of the origin, showing that $DH$ has unbounded variation on any interval $[0,\delta]$ is equivalent to showing that the {\em affine derivative} $D \log DH = D^2 H / D H$ does not belong to $L^1 ([0,\delta])$. This is easier to handle. Indeed, using the cocycle relation of this expression, we obtain
$$\frac{D^2 H}{D H} = \frac{D^2 h_{\alpha}}{D h_{\alpha}} - \frac{D^2 h_{\beta}}{D h_{\beta}} \circ H \cdot DH.$$
Assume that $\alpha > \beta$ (the case of a reverse inequality is analogous). Then,
$$\frac{D^2 h_{\alpha}}{D h_{\alpha}} = - Ds_{\alpha} +\frac{D u_{\alpha}}{1 + u_{\alpha}} 
= - \alpha \,  Ds + \frac{D u_{\alpha}}{1 + u_{\alpha}} .$$
By conditions (ii) and (iii) of Lemma \ref{l:conditions-ac}, the fraction on the right belongs to $L^1 ([0,\delta])$. 
Thanks to Lemma \ref{l:estimate-D2D}, this allows to conclude that 
$$\int_{\varepsilon}^{\delta} \left| \frac{D^2 h_{\alpha}}{D h_{\alpha}} \right|  \sim  \frac{ 2\, \alpha}{\pi} \, \log (\log (|\log (\varepsilon)|)).$$
Similarly, since $H$ is a $C^1$ diffeomorphism with $DH(0) =1$, we have 
$$\int_{H(\varepsilon)}^{H(\delta)} \left| \frac{D^2 h_{\beta}}{D h_{\beta}} \right|  
\sim  \frac{ 2\, \beta}{\pi} \, \log (\log (|\log (H(\varepsilon))|)) 
\sim  \frac{ 2\, \beta}{\pi} \, \log (\log ( | \log (\varepsilon)|)).$$
The inequality 
$$\int_{\varepsilon}^{\delta} \left| \frac{D^2 H}{D H} \right| 
\geq 
\int_{\varepsilon}^{\delta} \left| \frac{D^2 h_{\alpha}}{D h_{\alpha}} \right| 
- \int_{\varepsilon}^{\delta} \left| \frac{D^2 h_{\beta}}{D h_{\beta}} \circ H \cdot DH \right| 
= 
\int_{\varepsilon}^{\delta} \left| \frac{D^2 h_{\alpha}}{D h_{\alpha}} \right| 
- \int_{H(\varepsilon)}^{H(\delta)} \left| \frac{D^2 h_{\beta}}{D h_{\beta}}  \right| $$
then shows that $ \int_{\varepsilon}^{\delta} \left| \frac{D^2 H}{D H} \right|$ is of order \, 
$2 \, (\alpha - \beta) \, \log (\log ( |\log(\varepsilon)|)) / \pi$, \, which implies that $D^2 H / DH$ does not belong to $L^1([0,\delta])$. This closes the proof.
\end{proof}

\begin{rem}
\label{weak} 
Unfortunately, the perturbation methods used to prove Theorems \ref{t:A} and \ref{t:B} fail in the $C^{1+ac}$ setting when one starts with an arbitrary hyperbolic vector field. These seem to work, however, for  a vector field $X$ for which $DX$ has a certain logarithmic modulus of continuity for the derivative. We did not pursue the computations since we did not find a strategy for the general case. We thus leave open the question below.
\end{rem}

\begin{qs} 
Does the $C^1$ conjugacy class of {\em every} germ of hyperbolic $C^{1+ac}$ vector field contain uncountably many $C^{1+ac}$ conjugacy classes ? 
\end{qs}

\vspace{0.2cm}

%%%%%%%%%%%%%%%%%%%%%%%%%%%%%%%%%%%%%%%%%%%%%%%%%%%%%%%%%%%%%%%%%%%%%%%
%%%%%%%%%%%%%%%%%%%%%%%%%%%%%%%%%%%%%%%%%%%%%%%%%%%%%%%%%%%%%%%%%%%%%%%

\section*{Declarations}
 
\subsection*{Ethical Approval }
Non applicable

\subsection*{Competing interests} 
The authors have no relevant financial or non-financial interests to disclose.

\subsection*{Authors' contributions }
Both authors have contributed equally to the work.

\subsection*{Funding}
Hélène Eynard-Bontemps was partially funded by CNRS (Centre National de Recherche Scientifique, France), IRGA project ADMIN (Grenoble INP - Université Grenoble Alpes), Centro de Modelamiento Matemático (CMM, FB210005, BASAL funds for centers of excellence from ANID-Chile).
Andrés Navas was supported by Fondecyt research project 1220032.
 
\subsection*{Availability of data and materials}
Non applicable

%%%%%%%%%%%%%%%%%%%%%%%%%%%%%%%%%%%%%%%%%%%%%%%%%%%%

\begin{footnotesize}

\vspace{0.42cm}

\noindent {\bf H\'el\`ene Eynard-Bontemps} \hfill{\bf Andr\'es Navas}

\noindent Institut Fourier \hfill{ Dpto. de Matem\'aticas y C.C.}

\noindent  Universit\'e Grenoble Alpes \hfill{ Universidad de Santiago de  Chile}

\noindent 100 rue des Math\'ematiques \hfill{Alameda Bernardo O'Higgins 3363}

\noindent 38610 Gi\`eres, France \hfill{Estaci\'on Central, Santiago, Chile} 

\noindent helene.eynard-bontemps@univ-grenoble-alpes.fr \hfill{andres.navas@usach.cl}

\end{footnotesize}
\end{document}